\documentclass{amsart}
\usepackage{graphicx}
\usepackage{amsmath,amssymb}
\usepackage{amsthm}

\newtheorem{thm}{Theorem}[section]
\newtheorem{cor}[thm]{Corollary}
\newtheorem{lem}[thm]{Lemma}
\newtheorem{prop}[thm]{Proposition}

\theoremstyle{definition}
\newtheorem{defn}[thm]{Definition}
\theoremstyle{remark}
\newtheorem{rem}[thm]{Remark}
\theoremstyle{example}
\newtheorem{exa}[thm]{Example}
\theoremstyle{conjecture}

\numberwithin{equation}{section}

\newcommand{\im}{{\rm Im}\,}
\newcommand{\re}{{\rm Re}\,}

\newcommand{\calA}{\mathcal A}
\newcommand{\calC}{\mathcal C}
\newcommand{\calF}{\mathcal F}
\newcommand{\calH}{\mathcal H}

\newcommand{\calJ}{\mathcal J}

\newcommand{\calQ}{\mathcal Q}

\newcommand {\C} {\mathbb C}
\newcommand {\R} {\mathbb R}
\newcommand {\X} {\mathbb X}

\newcommand {\N} {\mathbb N}
\newcommand {\D} {\mathbb D}

\begin{document}

\title[Unbounded Weighted Composition Operators]{Unbounded Weighted Composition Operators on Fock space}%

\date{\today}%

\author{Pham Viet Hai}%
\address[P. V. Hai]{ISE department at the National University of Singapore, 117576, Singapore.}%
\email{isepvh@nus.edu.sg}




\maketitle

\begin{abstract}
In this paper, we consider \emph{unbounded} weighted composition operators acting on Fock space, and investigate some important properties of these operators, such as $\calC$-selfadjoint (with respect to weighted composition conjugations), Hermitian, normal, and cohyponormal. In addition, the paper shows that unbounded normal weighted composition operators are contained properly in the class of $\calC$-selfadjoint operators with respect to weighted composition conjugations.
\end{abstract}

\section{Introduction}

\subsection{Complex symmetric operators}
In their papers \cite{GP1, GP2}, Garcia and Putinar undertook the general study of complex symmetric operators with many motivations coming from function theory, matrix analysis and other areas. A number of other authors have recently made significant contributions to theory as well as applications in quantum mechanics (see e.g. \cite{GPP}).

To proceed, we first recall some terminologies. Let $\calH$ be a separable complex Hilbert space endowed with inner product $\langle .,.\rangle$. The domain of an unbounded linear operator is denoted as $\text{dom}(\cdot)$. For two unbounded linear operators $F,G$, the notation $F\preceq G$ means that $G$ is an \emph{extension} of $F$ (see \cite[Section 1.1]{KS}). Furthermore, if $A,B$ are two bounded linear operators on $\calH$, then we define the operator $AFB$ by
$$
\text{dom}(AFB):=\{f\in\calH:Bf\in\text{dom}(F)\},\quad (AFB)f:=AF(Bf).
$$
Note that we also use this notation in the case when $A,B$ are anti-linear.

\begin{defn}\label{alin-conj} An anti-linear mapping $\calC\colon\calH\to\calH$ is called a \emph{conjugation}, if it is both involutive and isometric.
\end{defn}

\begin{defn}\label{defcso-new}
Let $S\colon\text{dom}(S)\subseteq\calH \to \calH$ be a closed, densely defined, linear operator and $\calC$ a conjugation. We say that the operator $S$ is \emph{$\calC$-symmetric} if $S\preceq\calC S^*\calC$, and  \emph{$\calC$-selfadjoint} if $S=\calC S^*\calC$.

In both cases, the unbounded operator $S$ is \emph{complex symmetric}, in the precise sense
$$ [Sx,y] = [x,Sy],\quad\forall x,y \in \text{dom}(S),$$
where $[\cdot,\cdot]$ is the complex bilinear symmetric form induced by the conjugation $\calC$ and the inner product $\langle \cdot,\cdot\rangle$, namely
$$ [x,y] := \langle x, {\calC} y \rangle,\quad\forall x,y \in {\calH}.$$
\end{defn}

Clearly, a $\calC$-sefadjoint operator is always $\calC$-symmetric, but the converse statement can be wrong. In fact, the class of $\calC$-symmetric operators is much larger: it contains properly $\calC$-sefadjoint operators. In contrast to the usual symmetry, a $\calC$-symmetric operator always has a $\calC$-selfadjoint extension \cite{IMG1, IMG2} (see also \cite{IK, DR}).

The first step toward understanding a complex symmetric operator is to determine its internal structure. An effective method to this problem is to characterize which special operators are complex symmetric. Through a series of works, many well-known operators, such as Hermitian operators, unitary operators, and normal operators, have been proved to belong to this class.

\subsection{Linear weighted composition operator}
Among well-known operators, the class of weighted composition operators can connect basic questions about linear operators to classical results from the theory of holomorphic functions (see e.g. \cite{CM}). Considered on function spaces, weighted composition operators provide new meanings to classical theorems (such as boundedness, compactness, closed graph, etc.). These operators form today a vast chapter of modern analysis, and they are defined as follows.
\begin{defn}
Let $\X$ be a Hilbert space of holomorphic functions on a domain set $U\subseteq\C$. We consider formal \emph{weighted composition expressions} of the form
$$
E(\psi,\varphi)f=\psi\cdot f\circ\varphi,
$$
where $\psi:U\to\C$, $\varphi:U\to U$ are holomorphic functions. We are concerned with the operators arising from the formal expression $E(\psi,\varphi)$ in $\X$. The following operator is called the \emph{maximal weighted composition operator} on $\X$ corresponding to the expression $E(\psi,\varphi)$:
$$
\text{dom}(W_{\psi,\varphi,\max})=\{f\in\X:E(\psi,\varphi)f\in\X\},
$$
$$
W_{\psi,\varphi,\max}f=E(\psi,\varphi)f,\quad\forall f\in\text{dom}(W_{\psi,\varphi,\max}).
$$
The domain $\text{dom}(W_{\psi,\varphi,\max})$ is called the \emph{maximal domain}.
\end{defn}

The operator $W_{\psi,\varphi,\max}$ is ``maximal" in the sense that it cannot be extended as an operator in $\X$ generated by the expression $E(\psi,\varphi)$. It should be noted that the domain is crucial for an unbounded operator. The same formal expression considered on different domains may generate operators with completely different properties. This note suggests to consider the weighted composition expressions on subspaces of the maximal domain.

\begin{defn}
The operator $W_{\psi,\varphi}$ is called an \emph{unbounded weighted composition operator} if $W_{\psi,\varphi}\preceq W_{\psi,\varphi,\max}$. In this situation, the domain $\text{dom}(W_{\psi,\varphi})$ is a subspace of $\text{dom}(W_{\psi,\varphi,\max})$, and the operator $W_{\psi,\varphi}$ is the restriction of the maximal operator $W_{\psi,\varphi,\max}$ on $\text{dom}(W_{\psi,\varphi})$.
\end{defn}


In order to call for investigations, Garcia and Hammond published the paper \cite{GH} with the title being an open question to the community of operator theory: ``Which weighted composition operators are complex symmetric?". These authors in \cite{GH} and Jung et al. in \cite{JKKL} explored independently an internal structure of \emph{bounded} complex symmetric weighted composition operators acting on Hardy spaces in the unit disk $\D$ with respect to the conjugation
\begin{equation}\label{conjugation-J}
\calJ f(z)=\overline{f(\overline{z})}.
\end{equation}
Later, in \cite{WY}, some of the results were extended to Hardy spaces in the unit ball, with the same kind conjugation (but defined in higher dimensions).

The structure of the conjugation $\calJ$ inspired the author to study in \cite{HK1} a generalization, namely \emph{anti-linear weighted composition operators} $\calA_{\xi,\eta}f=\xi\cdot\overline{f\circ\overline{\eta}}$
acting on the Fock space.

\subsection{Fock space}
Recall that the \emph{Fock space} $\calF^2$ consists of entire functions which are square integrable with respect to the Gaussian measure $\frac{1}{\pi}e^{-|z|^2}\;dV(z)$, where $dV$ is Lebesgue measure on $\C$. This is a functional Hilbert space, with the inner product and kernel functions given by
$$
\langle f,g\rangle = \frac{1}{\pi}\int_\C f(z)\overline{g(z)}e^{-|z|^2}\;dV(z),
$$
$$K_z^{[m]}(u)=u^m e^{\overline{z}u},\quad z,u\in\C,m\in\N,$$
respectively. Since $\|K_z\|=e^{|z|^2/2}$, we always have
\begin{equation}\label{point}
|f(z)|=|\langle f,K_z \rangle| \leq \|f\|e^{|z|^2/2},\quad\forall f\in\calF^2.
\end{equation}
Thus, convergence in the norm of $\calF^2$ implies a point convergence. For more information about Fock spaces, we refer the reader to the monograph \cite{KZ}.

A characterization of anti-linear weighted composition operators, which are conjugations on $\calF^2$ was given in \cite{HK1}. These operators are called as \emph{weighted composition conjugations}, and they are described as follows. For complex numbers $a,b,c$ satisfying
\begin{equation}\label{abc-cond}
|a|=1, \quad \bar{a}b+\bar{b}=0, \quad |c|^2e^{|b|^2}=1,
\end{equation}
the weighted composition conjugation is defined by
\begin{equation}\label{wcc-jormulas}
\calC_{a,b,c}f(z):=ce^{bz}\overline{f\left(\overline{az+b}\right)},\quad\forall f\in\calF^2.
\end{equation}
The class of weighted composition conjugations contains the conjugation $\calJ$ defined by \eqref{conjugation-J} as a very particular case.

In \cite{HK1}, the author characterized all bounded weighted composition operators, which are complex symmetric with respect to weighted composition conjugations. Naturally, one is also interested in determining whether there are any additional classes of \emph{unbounded} complex symmetric weighted composition operators on $\calF^2$.

\subsection{Content}
This paper investigates some important properties of \emph{unbounded} weighted composition operators on Fock space $\calF^2$, such as $\calC$-selfadjoint (with respect to weighted composition conjugations), Hermitian, normal, and cohyponormal. 

The rest of this paper is organized as follows. Section \ref{s2} is devoted to recalling basic properties of bounded weighted composition operators on $\calF^2$. We consider an unbounded weighted composition operator $W_{\psi,\varphi}$, and prove auxiliary results in Section 3. Theorem \ref{varphi-cohypo} shows that under certain conditions, the symbol $\varphi$ is affine, while Theorem \ref{ad-2} provides the concrete structure of the adjoint $W_{\psi,\varphi}^*$ in the case when $\psi$ is an exponential form and $\varphi$ is affine. In Sections \ref{s4}-\ref{s6}, we characterize \emph{maximal weighted composition operators}, which are $\calC_{a,b,c}$-selfadjoint (Theorem \ref{<->cs}), Hermitian (Theorem \ref{hermitian-op}), normal (Theorem \ref{cohypo-nor-equiv-}), cohyponormal (Theorem \ref{cohypo-nor-equiv}),  respectively. In addition, the study of \emph{unbounded weighted composition operators} with arbitrary domains is also carried out in Theorems \ref{abc-selfadjointness}, \ref{self-arbitrary-domain}, \ref{cohypo-nor-equiv--}, and \ref{normal-2-arbi}. It should be emphasized that the class of complex symmetric operators obtained here contains operators studied in \cite{HK1} as a proper subclass. Furthermore, it also includes properly unbounded normal weighted composition operators (Corollary \ref{cohypo-nor-equiv-cor}).

\section{Preliminaries}\label{s2}
Let $W_{\psi,\varphi}$ be a unbounded linear weighted composition operator, induced by two entire functions $\psi$, $\varphi$. In the whole paper, we always assume that $\psi\not\equiv0$. 

It is clear that the domain $\text{dom}(W_{\psi,\varphi})$ is a non-empty subspace of $\calF^2$ (since $0\in\text{dom}(W_{\psi,\varphi})$). In general, $\text{dom}(W_{\psi,\varphi})$ is a proper subspace of $\calF^2$. To give an example for this claim, we make use of the following useful lemma.

\begin{lem}[\cite{KHI}]\label{KHI-lem}
If the function $h$ is entire, then $e^h\in\calF^2$ if and only if $h(z)=\alpha z^2+\beta z+\gamma$ with $|\alpha|<1/2$.
\end{lem}

\begin{exa}
Let $\varphi(z)=4z$, $\psi(z)=e^{2z}$, and $f(z)=e^{z^2/4}$. Then by Lemma \ref{KHI-lem}, $f\in\calF^2$, while $E(\psi,\varphi)f(z)=e^{4z^2+2z}\notin\calF^2$, that is $f\notin\text{dom}(W_{\psi,\varphi})$.
\end{exa}

\begin{rem}\label{f-in-dom}
Note that a function $f\in\calF^2$ belongs to the domain $\text{dom}(W_{\psi,\varphi,\max})$ if and only if $\psi\cdot f\circ\varphi\in\calF^2$, or equivalently if and only if
$$
\int_{\C}|\psi(z)f(\varphi(z))|^2 e^{-|z|^2}\;dV(z)<\infty.
$$
\end{rem}

A characterization of weighted composition operators, which are bounded on $\calF^2$ was carried out in \cite{TL}, where the techniques of adjoint operators in Hilbert spaces play a key role in proving the necessity. In \cite{HK2}, the author and Khoi used a different approach (not using the adjoint operator) to characterize the boundedness of weighted composition operators acting on the more general Fock spaces. In particular, the following illustrative example was given.
\begin{prop}[\cite{HK2}]\label{exaa}
Let $\varphi(z)=Az+B$, $\psi(z)=Ce^{Dz}$, where $A,B,C,D$ are complex constants. Then the operator $W_{\psi,\varphi}$ is bounded on $\calF^2$ if and only if
    \begin{enumerate}
    \item either $|A|<1$,
    \item or $|A|=1$, $D+A\overline{B}=0$.
    \end{enumerate}
\end{prop}

As mentioned in the Introduction, the authors characterized in \cite{HK1} all bounded weighted composition operators, which are $\calC_{a,b,c}$-symmetric on $\calF^2$.
\begin{prop}[\cite{HK1}]\label{cri-C-sym}
Let $\calC_{a,b,c}$ be a weighted composition conjugation, and $W_{\psi,\varphi}$ a bounded weighted composition operator induced by two entire functions $\psi,\varphi$. Then $W_{\psi, \varphi}$ is $\calC_{a,b,c}$-symmetric if and only if the following conditions hold:
\begin{itemize}
\item[(i)] $\varphi(z)=Az+B$, $\psi(z) = Ce^{Dz}$, with $C\ne 0,\ D=aB-bA+b.$
\item[(ii)] Either $|A| < 1$, or $|A|=1$, $D+A\overline{B}=0$.
\end{itemize}
\end{prop}

It is worth to mention a standard technique when one characterizes the complex symmetry of bounded operators. Recall that a bounded operator which is complex symmetric on a dense subset, is necessarily complex symmetric on the whole Hilbert space. Thus, Proposition \ref{exaa} (a criteria for boundedness) plays an indispensable role in proving the sufficient condition of Proposition \ref{cri-C-sym}.

\section{Some initial properties}
This section contains several auxiliary results which will be used to prove the main results. Some of these results may have an intrinsic value.

\subsection{Reproducing kernels}
The first observation is concerned with the action of an unbounded weighted composition operator on the kernel functions. It allows us to predict a form for eigenvalues of $W_{\psi,\varphi}^*$ when the symbol $\varphi$ has a fixed point.
\begin{lem}\label{W*Kz-prop}
Let $W_{\psi,\varphi}$ be a densely defined unbounded weighted composition operator induced by two entire functions $\psi,\varphi$. Then 
\begin{enumerate}
\item For every $z\in\C$, we always have $K_z\in\text{dom}(W_{\psi,\varphi}^*)$, and
\begin{equation*}
W_{\psi,\varphi}^*K_z=\overline{\psi(z)}K_{\varphi(z)}.
\end{equation*}
\item In particular, if $\varphi(z)=Az+B$, where $A,B$ are constants, then for every $z\in\C$, $m\in\N$, $K_z^{[m]}\in\text{dom}(W_{\psi,\varphi}^*)$, and
\begin{equation*}
W_{\psi,\varphi}^*K_z^{[m]}=\sum_{j=0}^m\binom{m}{j}\overline{\psi^{(m-j)}(z)A^{j}}K_{Az+B}^{[j]}.
\end{equation*}
\end{enumerate}
\end{lem}
\begin{proof}
(1) For every $f\in\text{dom}(W_{\psi,\varphi})$, we have
\begin{eqnarray*}
\langle W_{\psi,\varphi}f,K_z\rangle &=& W_{\psi,\varphi}f(z)=\psi(z)\langle f,K_{\varphi(z)}\rangle=\langle f,\overline{\psi(z)}K_{\varphi(z)}\rangle,
\end{eqnarray*}
which gives conclusion (1).

(2) Now suppose that $\varphi(z)=Az+B$. Let $f\in\text{dom}(W_{\psi,\varphi})$. By induction, we can show that
$$
(f\circ\varphi)^{(\ell)}(z)=A^{\ell}f^{(\ell)}(Az+B),\quad\forall\ell\in\N,\forall z\in\C,
$$
and hence,
\begin{eqnarray*}
\langle W_{\psi,\varphi}f,K_z^{[m]}\rangle &=& (W_{\psi,\varphi}f)^{(m)}(z)=\sum_{j=0}^m\binom{m}{j}\psi^{(m-j)}(z)A^{j}f^{(j)}(Az+B).
\end{eqnarray*}
Since $f^{(j)}(Az+B)=\langle f,K_{Az+B}^{[j]}\rangle$, we get
\begin{eqnarray*}
\langle W_{\psi,\varphi}f,K_z^{[m]}\rangle 
&=&\langle f,\sum_{j=0}^m\binom{m}{j}\overline{\psi^{(m-j)}(z)A^{j}}K_{Az+B}^{[j]}\rangle,
\end{eqnarray*}
which gives conclusion (2).
\end{proof}

The next result shows a structural description of the kernel of the operator $W_{\psi,\varphi}$ and hence the range of $W_{\psi,\varphi}^*$.

\begin{prop}\label{W-dense}
Let $W_{\psi,\varphi}$ be a densely defined unbounded weighted composition operator induced by two entire functions $\psi,\varphi$. If the function $\psi$ is nowhere vanished and $\varphi$ is non-constant, then
$$
\ker(W_{\psi,\varphi})=\{0\},\quad\overline{\im(W_{\psi,\varphi}^*)}=\calF^2(\C).
$$
\end{prop}
\begin{proof}
Let $f\in\ker(W_{\psi,\varphi})$. For every $z\in\C$, we have $\psi(z)f(\varphi(z))=0$, which gives $f(\varphi(z))=0$, and hence, $f\equiv 0$. Thus, $\ker(W_{\psi,\varphi})=\{0\}$.

Furthermore,
$$
\calF^2=\overline{\im(W_{\psi,\varphi}^*)}\oplus\ker(W_{\psi,\varphi})=\overline{\im(W_{\psi,\varphi}^*)}.
$$
\end{proof}

\begin{thm}\label{varphi-cohypo}
Let $W_{\psi,\varphi}$ be a densely defined unbounded weighted composition operator, induced by two entire functions $\psi\not\equiv 0$ and $\varphi\not\equiv\text{const}$. Suppose that there exists an involutive mapping $S:\calF^2\to\calF^2$, such that
\begin{equation}\label{SKz-dom}
\text{dom}(W_{\psi,\varphi}^*\big)\subseteq\text{dom}(W_{\psi,\varphi}S)
\end{equation}
and
\begin{equation}\label{WS<W*}
\|W_{\psi,\varphi}Sf\| \leq \|W_{\psi,\varphi}^*f\|,\ \forall f\in\text{dom}(W_{\psi,\varphi}^*).
\end{equation}
The following conclusions hold.
\begin{enumerate}
\item The function $\psi$ is never vanished. Furthermore, if $\psi\in\calF^2$, then it takes the form $\psi(z)=\psi(0)e^{C z^2+D z}$, where $C,D$ are constants with $|C|<1/2$ and $\psi(0)\ne 0$.
\item The function $\varphi$ takes the form $\varphi(z)=Az+B$, where $A,B$ are complex constants, with $A\ne 0$.
\item If $A=1$ and $S$ is the identity operator, then $\psi(z)=\psi(0)e^{Dz}$, where $D$ is a complex constant, and $\psi(0)\ne 0$, $|B|\geq |D|$.
\end{enumerate}
\end{thm}

\begin{proof}
(1) Assume in contrary that $\psi(z_0)=0$ for some $z_0\in\C$. Then there is a neighbourhood $V$ of $z_0$ such that $\psi(z)\ne0$ for every $z \in V\setminus\{z_0\}$. Lemma \ref{W*Kz-prop}(1) shows that $K_{z_0}\in\text{dom}(W_{\psi,\varphi}^*)$ and $W_{\psi,\varphi}^*K_{z_0} = \overline{\psi(z_0)}K_{\varphi(z_0)} = 0$.

By assumptions \eqref{SKz-dom}-\eqref{WS<W*}, we have $SK_{z_0}\in\text{dom}(W_{\psi,\varphi})$ and $W_{\psi,\varphi}SK_{z_0}=0$. Consequently, taking into account the structure of the operator $W_{\psi,\varphi}$, we have
$$
\psi(z)SK_{z_0}(\varphi(z))=W_{\psi, \varphi}SK_{z_0}(z)=0,\quad\forall z\in\C,
$$
which implies that $SK_{z_0}\circ\varphi\equiv 0$ on $V\setminus\{z_0\}$. Since $\varphi$ is a non-constant function, $SK_{z_0}\equiv 0$, and hence, $K_{z_0}\equiv 0$ (because $S$ is involutive). But it is impossible.

The rest part of this conclusion follows from Lemma \ref{KHI-lem}.

(2) By \cite[Exercise 14, Chapter 3]{SS}, it is enough to show that the function $\varphi$ is injective.

Suppose that $\varphi(z_1)=\varphi(z_2)$, for some $z_1,z_2\in\C$. Since $K_{z_1}$ and $K_{z_2}$ both belong to the domain $\text{dom}(W_{\psi,\varphi}^*)$, so do their linear combinations. Lemma \ref{W*Kz-prop}(1) gives
$$
W_{\psi,\varphi}^*\big(\overline{\psi(z_2)}K_{z_1}-\overline{\psi(z_1)}K_{z_2}\big) = \overline{\psi(z_1)\psi(z_2)}K_{\varphi(z_1)}-\overline{\psi(z_1)\psi(z_2)}K_{\varphi(z_2)}=\mathbf{0},
$$
which implies, again by assumption \eqref{WS<W*}, that $W_{\psi,\varphi}S(\overline{\psi(z_2)}K_{z_1}-\overline{\psi(z_1)}K_{z_2})=\mathbf{0}$. This means that $S(\overline{\psi(z_2)}K_{z_1}-\overline{\psi(z_1)}K_{z_2})\in\ker(W_{\psi,\varphi})$, and hence, by Proposition \ref{W-dense}, it must be a zero function. Since the operator $S$ is involutive, we get
$$
(\overline{\psi(z_2)}K_{z_1}-\overline{\psi(z_1)}K_{z_2})(u)=0,\ \forall u\in\C,
$$
which gives $z_1=z_2$.

(3) Now suppose that $\varphi(z)=z+B$ and $S$ is the identity operator. By Lemma \ref{W*Kz-prop}(1) and assumption \eqref{WS<W*}, we have
\begin{eqnarray}\label{WK<W*K}
|\psi(z)|\cdot\|K_{\varphi(z)}\|&=&\|W_{\psi,\varphi}^*K_z\|\geq \|W_{\psi,\varphi}K_z\|\nonumber\\
&\geq& |\langle W_{\psi,\varphi}K_z,K_u\rangle|\cdot\|K_u\|^{-1}\nonumber\\
&=&|W_{\psi,\varphi}K_z(u)|\cdot\|K_u\|^{-1}\nonumber\\
&=&|\psi(u)e^{\overline{z}\varphi(u)}|e^{-|u|^2/2}.
\end{eqnarray}
Since the function $\psi$ is nowhere vanished, we can rewrite the above as follows
\begin{equation}\label{A=1}
\left|\dfrac{\psi(u)}{\psi(z)}\right|e^{\re(\overline{z}\varphi(u))-|u|^2/2-|\varphi(z)|^2/2}\leq 1,\ \forall u,z\in\C.
\end{equation}
Note that by Lemma \ref{W*Kz-prop}(1) and assumption \eqref{SKz-dom}, we see
$$1=K_0\in\text{dom}(W_{\psi,\varphi}^*)\subseteq\text{dom}(W_{\psi,\varphi}S)=\text{dom}(W_{\psi,\varphi}),$$
and so, $\psi=W_{\psi,\varphi}1\in\calF^2$. Using conclusion (1), this function takes the form $\psi(z)=\psi(0)e^{Cz^2+Dz}$ with $|C|<1/2$, and hence,
$$
\left|\dfrac{\psi(u)}{\psi(z)}\right|=e^{\re[(u-z)(C(z+u)+D)]}.
$$
Since $\varphi(z)=z+B$, we have $2\re(\overline{z}\varphi(u))-|u|^2-|\varphi(z)|^2=-|z-u|^2-|B|^2$.

Substituting the above identities back into \eqref{A=1}, we get
\begin{equation}\label{4C}
|z-u|^2-2\re[(u-z)(C(z+u)+D)]+|B|^2\geq 0,\ \forall u,z\in\C.
\end{equation}

Assume in contrary that $C\ne0$. For
$$
u=\frac{1}{2}\left(\frac{2|B|+1-D}{C}+2|B|+1\right)\quad\hbox{and}\quad z=\frac{1}{2}\left(\frac{2|B|+1-D}{C}-2|B|-1\right),
$$
we have
$$
|z-u|^2-2\re[(u-z)(C(z+u)+D)]+|B|^2=(-|B|-1)(3|B|+1)<0,
$$
which contradicts \eqref{4C}. 

Thus, we must have $C=0$, and hence, inequality \eqref{4C} is reduced to
$$
|z-u|^2-2\re[(u-z)D]+|B|^2\geq 0,\ \forall u,z\in\C.
$$
In particular with $u-z=\overline{D}$, we obtain $|B|\geq |D|$, and the proof of the theorem is complete.
\end{proof}

\subsection{Dense domain and closed graph}
The following result may be well-known, but we give a proof, for a completeness of exposition.
\begin{prop}\label{W-closed}
Every maximal weighted composition operator is closed on Fock space $\calF^2$.
\end{prop}

\begin{proof}
Let $W_{\psi,\varphi,\max}$ be the maximal weighted composition operator induced by two entire functions $\psi,\varphi$.

Furthermore, let $(f_n)$ be a sequence of functions in $\calF^2$ and $f,g\in\calF^2$, such that
$$
f_n\to f \quad\text{and}\quad W_{\psi,\varphi,\max}f_n\to g\quad\hbox{in $\calF^2$}.
$$
By \eqref{point}, we have
$$
f_n(z)\to f(z) \quad\text{and}\quad W_{\psi,\varphi,\max}f_n(z)\to g(z),\quad\forall z\in\C,
$$
and so,
$$
W_{\psi,\varphi,\max}f_n(z)=\psi(z)f_n(\varphi(z))\to\psi(z)f(\varphi(z)),\quad\forall z\in\C.
$$
Thus,
$$
\psi(z)f(\varphi(z))=g(z),\quad\forall z\in \C.
$$
Since $g\in\calF^2$, we conclude that $f\in\text{dom}(W_{\psi,\varphi,\max})$ and $W_{\psi,\varphi,\max}f=g$.
\end{proof}

The result below offers an alternate description of the maximal weighted composition operators.
\begin{prop}\label{Wmax=J}
Let $\calQ$ be the linear operator given by
$$
\text{dom}(\calQ)=\text{Span}(\{K_z:z\in\C\}),\quad\calQ K_z=\overline{\psi(z)}K_{\varphi(z)}.
$$
Then $W_{\psi,\varphi,\max}=\calQ^*$. Moreover, the operator $W_{\psi,\varphi,\max}$ is densely defined if and only if the operator $\calQ$ is closable.
\end{prop}
\begin{proof}
Let $f=\sum_{j=1}^n\lambda_j K_{z_j}\in\text{dom}(\calQ)$. For every $g\in\calF^2$, we have
\begin{eqnarray*}
\langle \calQ f,g\rangle = \sum_{j=1}^n\lambda_j\langle \overline{\psi(z_j)}K_{\varphi(z_j)},g\rangle=\sum_{j=1}^n\lambda_j\overline{\psi(z_j)g(\varphi(z_j))}=\sum_{j=1}^n\lambda_j\overline{E(\psi,\varphi)g(z_j)}.
\end{eqnarray*}
Note that by the Riesz lemma, the function $g$ belongs to the domain $\text{dom}(\calQ^*)$ if and only if there exists $C>0$ such that
$$
|\langle \calQ f,g\rangle|\leq C\|f\|,\quad\forall f\in\text{dom}(\calQ),
$$
or equivalently, if and only if
$$
|\sum_{j=1}^n E(\psi,\varphi)g(z_j)\overline{\lambda_j}|^2\leq C^2\sum_{j,\ell=1}^n \lambda_j\overline{\lambda_\ell}K_{z_j}(z_\ell).
$$
In view of \cite{FHS}, the latter is equivalent to $E(\psi,\varphi)g\in\calF^2$. This shows that $\text{dom}(\calQ^*)=\text{dom}(W_{\psi,\varphi,\max})$. Moreover,
$$
\langle \calQ f,g\rangle=\langle f,E(\psi,\varphi)g\rangle=\langle f,W_{\psi,\varphi,\max}g\rangle,\quad\forall f\in\text{dom}(\calQ),\forall g\in\text{dom}(W_{\psi,\varphi,\max}),
$$
which gives $W_{\psi,\varphi,\max}=\calQ^*$.

The rest conclusion follows from \cite[Proposition 1.8(i)]{KS}.
\end{proof}

\subsection{Adjoints}

As it will be seen in the next section, for $\calC_{a,b,c}$-selfadjoint weighted composition operators, the symbol $\psi$ has an exponential form, while $\varphi$ is affine. Thus, it is worth to give an explicit formula for the adjoint $W_{\psi,\varphi,\max}^*$ on Fock space $\calF^2$.

The following simple note is useful for showing that two unbounded operators are equal.

\begin{lem}[{\cite[Lemma 1.3]{KS}}]\label{TS}
Let $T\colon\text{dom}(T)\to \X$, $S:\text{dom}(S)\to \X$ be two linear operators acting on a Banach space $\X$. If $T\preceq S$, $T$ is onto, and $S$ is one-to-one, then $T=S$.
\end{lem}

Lemma \ref{TS} is used to prove the following result.
\begin{thm}\label{ad-2}
Let $\psi(z)=Ce^{Dz}$, $\varphi(z)=Az+B$, $\widehat{\psi}(z)=\overline{C}e^{\overline{B}z}$, and $\widehat{\varphi}(z)=\overline{A}z+\overline{D}$, where $A,B,C$, and $D$ are complex constants, with $C\ne 0$. Then we always have $W_{\psi,\varphi,\max}^*= W_{\widehat{\psi},\widehat{\varphi},\max}$.
\end{thm}

\begin{proof} 
Note that a direct computation shows that for every $z\in\C$, $K_z\in\text{dom}(W_{\psi,\varphi,\max})$, and moreover, $W_{\psi,\varphi,\max}K_z=Ce^{B\overline{z}}K_{\overline{A}z+\overline{D}}$.

First, we show that
\begin{equation}\label{first-implication}
W_{\psi,\varphi,\max}^*\preceq W_{\widehat{\psi},\widehat{\varphi},\max}.
\end{equation}
Indeed, for every $f\in\text{dom}(W_{\psi,\varphi,\max}^*)$, we have
\begin{eqnarray*}
(W_{\psi,\varphi,\max}^*f)(z)%
&=&\langle W_{\psi,\varphi,\max}^*f,K_z \rangle=\langle f,W_{\psi,\varphi,\max}K_z \rangle\\
&=&\overline{C}e^{\overline{B}z}\langle f,K_{\overline{A}z+\overline{D}} \rangle= E(\widehat{\psi},\widehat{\varphi})f(z).
\end{eqnarray*}
So, $E(\widehat{\psi},\widehat{\varphi})f=W_{\psi,\varphi,\max}^*f\in\calF^2$, which shows that $f\in\text{dom}(W_{\widehat{\psi},\widehat{\varphi},\max})$ and $W_{\psi,\varphi,\max}^*f=W_{\widehat{\psi},\widehat{\varphi},\max}f$.

Next, we prove the equality of \eqref{first-implication} occurs. There are three possibilities for $|A|$.

{\bf - Case 1:} $|A|<1$. 

By Proposition \ref{exaa}, the operator $W_{\psi,\varphi,\max}$ is bounded. Then the desired result follows from \eqref{first-implication}.

{\bf - Case 2:} $|A|>1$. 

In this case, we make use of Lemma \ref{TS} (with $T=W_{\psi,\varphi,\max}^*$ and $S=W_{\widehat{\psi},\widehat{\varphi},\max}$). Note that by Proposition \ref{W-dense}, the operator $W_{\widehat{\psi},\widehat{\varphi},\max}$ is always one-to-one. 

Also by Proposition \ref{W-dense}, the range $\im W_{\psi,\varphi,\max}^*$ is dense in $\calF^2$. So, to show that $W_{\psi,\varphi,\max}^*$ is onto, we have to prove that the range $\im W_{\psi,\varphi,\max}^*$ is closed. For this, it is enough to show that there exists $L>0$ such that
\begin{equation}\label{WC-norm}
L\|W_{\psi,\varphi,\max}^*f\|\geq\|f\|,\quad\forall f\in\text{dom}(W_{\psi,\varphi,\max}^*).
\end{equation}
Indeed, setting $g=W_{\psi,\varphi,\max}^*f$, by \eqref{first-implication}, we also have $g=W_{\widehat{\psi},\widehat{\varphi},\max}f$. A direct computation gives $f=E(\xi,\eta)g$, where
$$
\xi(z)=\dfrac{1}{\overline{C}}e^{-\frac{\overline{B}}{\overline{A}}(z-\overline{D})},\quad\eta(z)=\dfrac{z-\overline{D}}{\overline{A}}.
$$
By Proposition \ref{exaa}, the operator $W_{\xi,\eta,\max}$ is bounded. Then there exists $M>0$ such that
$$
\|W_{\xi,\eta,\max}h\| \leq M\|h\|,\quad\forall h\in\calF^2.
$$
In particular, for $h=g$ we get $\|W_{\xi,\eta,\max}g\| \leq M\|g\|$. Since $W_{\xi,\eta,\max}g=E(\xi,\eta)g=f$ and $g=W_{\psi,\varphi,\max}^*f$, we obtain \eqref{WC-norm}.

{\bf - Case 3:} $|A|=1$.

By \eqref{first-implication}, it is enough to show that $\text{dom}(W_{\widehat{\psi},\widehat{\varphi},\max})\subseteq\text{dom}(W_{\psi,\varphi,\max}^*)$. Let $f\in\text{dom}(W_{\widehat{\psi},\widehat{\varphi},\max})$. By the Riesz lemma (see \cite[Section 1.2]{KS}), $f\in\text{dom}(W_{\psi,\varphi,\max}^*)$ if and only if there exists $M=M(f)>0$ such that
$$
|\langle W_{\psi,\varphi,\max}g,f\rangle | \leq M\|g\|,\quad\forall g\in\text{dom}(W_{\psi,\varphi,\max})\subseteq\calF^2.
$$
For this, we consider the following quantity
$$
\langle W_{\psi,\varphi,\max}g,f\rangle = \frac{1}{\pi}\int_{\C}Ce^{Dz}g(Az+B)\overline{f(z)} e^{-|z|^2}\;dV(z).
$$
Doing the change of variables $u=Az+B$, we have $z=\overline{A}u-\overline{A}B$ (since $|A|=1$), and hence, the integral above is equal to
$$
Ce^{-\overline{A}BD-|B|^2}\int_{\C}g(u)e^{-|u|^2/2}\overline{f(\overline{A}u-\overline{A}B)}e^{-|u|^2/2+\overline{A}Du+2\re (u\overline{B})}\;dV(u).
$$
We use the H\"older inequality to estimate
\begin{eqnarray*}
&&\int_{\C}\left|g(u)e^{-|u|^2/2}\overline{f(\overline{A}u-\overline{A}B)}e^{-|u|^2/2+\overline{A}Du+2\re (u\overline{B})}\right|\;dV(u)\\
&& \le
\pi^{1/2}\|g\|\cdot \left(\int_{\C}|f(\overline{A}u-\overline{A}B)|^2e^{-|u|^2+2\re(\overline{A}Du)+4\re (u\overline{B})}\;dV(u) \right)^{1/2}.
\end{eqnarray*}
Doing again the change of variables $\overline{A}u-\overline{A}B=\overline{A}v+\overline{D}$, we have $v=u-A\overline{D}-B$, and hence,
\begin{eqnarray*}
&&\int_{\C}|f(\overline{A}u-\overline{A}B)|^2e^{-|u|^2+2\re(\overline{A}Du)+4\re (u\overline{B})}\;dV(u)\\
&&=|C|^{-2}e^{-|A\overline{D}+B|^2+6\re(\overline{A}BD)+2|D|^2+4|B|^2}\int_{\C}|W_{\widehat{\psi},\widehat{\varphi},\max}f(v)|^2e^{-|v|^2}\;dV(v).
\end{eqnarray*}
Subsequently,
$$
|\langle W_{\psi,\varphi,\max}g,f\rangle|\leq |C|^{-1}e^{-|A\overline{D}+B|^2/2+2\re(\overline{A}BD)+|D|^2+|B|^2}\cdot\|g\|\cdot\|W_{\widehat{\psi},\widehat{\varphi},\max}f\|.
$$
The theorem is proved completely.
\end{proof}

\section{Complex symmetry}\label{s4}

\subsection{$\calC$-selfadjointness}

First we note that Proposition \ref{W-dense} and Theorem \ref{varphi-cohypo} (for $S=\calC$, the conjugation) can offer some properties of $\psi,\varphi$ when the operator $W_{\psi,\varphi}$ is $\calC$-selfadjoint with respect to an \emph{arbitrary conjugation}.

\begin{prop}\label{psi-ne-0}
Let $\calC$ be a conjugation on $\calF^2$, and $W_{\psi,\varphi}$ an unbounded $\calC$-selfadjoint weighted composition operator induced by two entire functions $\psi,\varphi$. Then the following conclusions hold.
\begin{enumerate}
\item The function $\psi$ is never vanished. Furthermore, if $\psi\in\calF^2$, then it takes the form $\psi(z)=\psi(0)e^{C z^2+D z}$, where $C,D$ are constants with $|C|<1/2$ and $\psi(0)\ne 0$.
\item The function $\varphi$ takes the form $\varphi(z)=Az+B$, where $A,B$ are complex constants, with $A\ne 0$.
\item The kernel $\ker W_{\psi,\varphi}=\{0\}$ and range $\im W_{\psi,\varphi}$ is dense in $\calF^2$.
\end{enumerate}
\end{prop}

\subsection{$\calC_{a,b,c}$-selfadjointness}

In this subsection, we give a complete description of \emph{unbounded} weighted composition operators, which are $\calC$-selfadjoint with respect to weighted composition conjugations (or simply: $\calC_{a,b,c}$-selfadjoint). The class of complex symmetric operators obtained here contains properly bounded operators investigated in the paper \cite{HK1}.

The following result is a necessary condition for \emph{maximal} weighted composition operators to be $\calC_{a,b,c}$-selfadjoint. 
\begin{prop}\label{c-sym-thm-nece}
Let $W_{\psi,\varphi,\max}$ be a maximal weighted composition operator induced by two entire functions $\psi,\varphi$. If
$$
\calC_{a,b,c}W_{\psi,\varphi,\max}^*K_z=W_{\psi,\varphi,\max}\calC_{a,b,c}K_z,\quad\forall z\in\C,
$$
 then the symbols are of the following forms
\begin{equation}\label{s-psi-2}
\varphi(z)=Az+B,\quad\psi(z) = Ce^{Dz},\ \hbox{with $C\ne 0,\ D=aB-bA+b$}.
\end{equation}
\end{prop}

\begin{proof}
Take arbitrarily $u,z\in\C$. On one hand, by Lemma \ref{W*Kz-prop}(1), we have
\begin{eqnarray*}
(\calC_{a,b,c}W_{\psi,\varphi,\max}^*K_z)(u)&=&\psi(z)(\calC_{a,b,c}K_{\varphi(z)})(u)=\psi(z)ce^{bu+\varphi(z)(au+b)}.
\end{eqnarray*}
On the other hand,
\begin{eqnarray*}
(W_{\psi,\varphi,\max}\calC_{a,b,c}K_z)(u)&=&\psi(u)(\calC_{a,b,c}K_z)(\varphi(u))=\psi(u)ce^{b\varphi(u)+z(a\varphi(u)+b)}.
\end{eqnarray*}
Thus, we obtain
\begin{equation}\label{WC=CW*}
\psi(z)ce^{bu+\varphi(z)(au+b)}=\psi(u)ce^{b\varphi(u)+z(a\varphi(u)+b)},\ \forall u,z\in\C.
\end{equation}
In particular, for $u=0$, we get
$$
\psi(z)e^{b\varphi(z)}=\psi(0)e^{b\varphi(0)+z(a\varphi(0)+b)},
$$
which gives
\begin{equation}\label{form-psi}
\psi(z)=\psi(0)e^{z(a\varphi(0)+b)-b\varphi(z)+b\varphi(0)}.
\end{equation}
Then \eqref{WC=CW*} is reduced to
\begin{equation}\label{form-vphi}
z\varphi(0)+\varphi(z)u=u\varphi(0)+z\varphi(u),\ \forall u,z\in\C.
\end{equation}
For all $u,z\in\C\setminus\{0\}$, we have
$$
\frac{\varphi(u)-\varphi(0)}{u}=\frac{\varphi(z)-\varphi(0)}{z},
$$
and hence
$$
\frac{\varphi(z)-\varphi(0)}{z}=A\in\C.
$$
Thus, $\varphi(z)=Az+B$, where $B=\varphi(0)$. 

Finally, substituting $\varphi$ into \eqref{form-psi}, we obtain \eqref{s-psi-2}.
\end{proof}

The next proposition makes precise the expression $\calC_{a,b,c}E(\psi,\varphi)\calC_{a,b,c}$, and hence, we obtain an explicit description of the operator $\calC_{a,b,c}W_{\psi,\varphi,\max}\calC_{a,b,c}$.

\begin{prop}\label{W-bar-CWC}
Let $\varphi(z)=Az+B$, $\psi(z) = Ce^{Dz}$, $\widehat{\varphi}(z)=\overline{A}z+\overline{D},\quad\widehat{\psi}(z)=\overline{C}e^{\overline{B}z}$, where $C\ne 0,\ D=aB-bA+b$. Then the following conclusions hold.
\begin{enumerate}
\item $\calC_{a,b,c}E(\psi,\varphi)\calC_{a,b,c}=E(\widehat{\psi},\widehat{\varphi})$.
\item $\calC_{a,b,c}W_{\psi,\varphi,\max}\calC_{a,b,c}=W_{\widehat{\psi},\widehat{\varphi},\max}$.
\end{enumerate}
\end{prop}

\begin{proof}
It is clear that conclusion (2) follows from conclusion (1).

We prove conclusion (1) as follows. For any $f\in\calF^2$, we have
\begin{eqnarray*}
E(\psi,\varphi)\calC_{a,b,c}f(z)
&=&cCe^{(bA+D)z+bB}\overline{f\left(\overline{aAz+aB+b}\right)},
\end{eqnarray*}
and hence,
\begin{eqnarray*}
\calC_{a,b,c}E(\psi,\varphi)\calC_{a,b,c}f(z)%
&=&|c|^2\overline{C}e^{\overline{bA+D}(az+b)+\overline{bB}+bz}f(\overline{aA}(az+b)+\overline{aB+b}).
\end{eqnarray*}
Note that since $a,b,c$ satisfy condition \eqref{abc-cond}, we have the following identities
$$
\overline{bA+D}(az+b)+\overline{bB}+bz=\overline{B}z+|b|^2,
$$
and
$$
\overline{aA}(az+b)+\overline{aB+b}=\overline{A}z+\overline{D}.
$$
Thus,
$$
\calC_{a,b,c}E(\psi,\varphi)\calC_{a,b,c}f(z)
= |c|^2\overline{C}e^{\overline{B}z+|b|^2}f(\overline{A}z+\overline{D})
=\overline{C}e^{\overline{B}z}f(\overline{A}z+\overline{D})
= E(\widehat{\psi},\widehat{\varphi})f(z).
$$
\end{proof}

With all preparation in place, we can now state and prove the main result of the present section. It turns out that condition \eqref{s-psi-2} is also sufficient for a maximal weighted composition operator to be $\calC_{a,b,c}$-selfadjoint. 

\begin{thm}[$\calC_{a,b,c}$-selfadjoint criterion]\label{<->cs}
Let $\calC_{a,b,c}$ be a weighted composition conjugation, and $W_{\psi,\varphi,\max}$ a maximal weighted composition operator induced by two entire functions $\psi$, $\varphi$ with $\psi\not\equiv 0$. Then the following assertions are equivalent.
\begin{enumerate}
\item The operator $W_{\psi,\varphi,\max}$ is $\calC_{a,b,c}$-selfadjoint.
\item The operator $W_{\psi,\varphi,\max}$ is densely defined and it satisfies
$$W_{\psi,\varphi,\max}^*\preceq\calC_{a,b,c}W_{\psi,\varphi,\max}\calC_{a,b,c}.$$
\item The symbols are of forms \eqref{s-psi-2}, that is
\begin{equation*}
\varphi(z)=Az+B,\quad\psi(z) = Ce^{Dz},\ \hbox{with $C\ne 0,\ D=aB-bA+b$}.
\end{equation*}
\end{enumerate}
\end{thm}

\begin{proof}
It is clear that $(1)\Longrightarrow(2)$, while implication $(2)\Longrightarrow(3)$ follows from Proposition \ref{c-sym-thm-nece}.

It remains to verify $(3)\Longrightarrow(1)$. Indeed, suppose that assertion (3) holds. By Proposition \ref{W-closed}, the operator $W_{\psi,\varphi,\max}$ is closed. A direct computation shows that kernel functions belong to the domain $\text{dom}(W_{\psi,\varphi,\max})$. Furthermore, by Theorem \ref{ad-2} and Proposition \ref{W-bar-CWC}, we have
$$
W_{\psi,\varphi,\max}^*= W_{\widehat{\psi},\widehat{\varphi},\max}=\calC_{a,b,c}W_{\psi,\varphi,\max}\calC_{a,b,c}.
$$
\end{proof}

In comparison with the case of bounded operators, the unbounded case uses more complicated techniques concerning the domains as well as adjoint operators. In addition, the fact that ``a bounded operator which is complex symmetric on polynomials, is necessarily complex symmetric on the whole $\calF^2$'' is no longer true for the unbounded case.

In the following result, we show that there is no non-trivial domain for an unbounded weighted composition operator $W_{\psi,\varphi}$ on which $W_{\psi,\varphi}$ is $\calC_{a,b,c}$-selfadjoint.

\begin{thm}\label{abc-selfadjointness}
Let $W_{\psi,\varphi}$ be an unbounded weighted composition operator, induced by the symbols $\psi$, $\varphi$ with $\psi\not\equiv 0$. Furthermore, let $\calC_{a,b,c}$ be a weighted composition conjugation. Then the operator $W_{\psi,\varphi}$ is $\calC_{a,b,c}$-selfadjoint if and only if the following conditions hold.
\begin{enumerate}
\item $W_{\psi,\varphi}=W_{\psi,\varphi,\max}$.
\item The symbols are of forms \eqref{s-psi-2}, that is
\begin{equation*}
\varphi(z)=Az+B,\quad\psi(z) = Ce^{Dz},\ \hbox{with $C\ne 0,\ D=aB-bA+b$}.
\end{equation*}
\end{enumerate}
\end{thm}
\begin{proof}
The sufficiency follows from Theorem \ref{<->cs}.

For the necessity, we suppose that $W_{\psi,\varphi}=\calC_{a,b,c}W_{\psi,\varphi}^*\calC_{a,b,c}$. First, we show that the operator $W_{\psi,\varphi,\max}$ is $\calC_{a,b,c}$-selfadjoint.

Since $W_{\psi,\varphi}\preceq W_{\psi,\varphi,\max}$, we have
$$
W_{\psi,\varphi,\max}^*\preceq W_{\psi,\varphi}^*=\calC_{a,b,c}W_{\psi,\varphi}\calC_{a,b,c}\preceq\calC_{a,b,c}W_{\psi,\varphi,\max}\calC_{a,b,c},
$$
which implies, due to the involutivity of $\calC_{a,b,c}$, that
$$
\calC_{a,b,c}W_{\psi,\varphi,\max}^*\preceq W_{\psi,\varphi,\max}\calC_{a,b,c}.
$$
Lemma \ref{W*Kz-prop} shows that kernel functions always belong to the domain $\text{dom}(\calC_{a,b,c}W_{\psi,\varphi,\max}^*)$, and so,
$$
\calC_{a,b,c}W_{\psi,\varphi,\max}^*K_z=W_{\psi,\varphi,\max}\calC_{a,b,c}K_z,\quad\forall z\in\C.
$$
By Proposition \ref{c-sym-thm-nece}, the symbols are of forms \eqref{s-psi-2}, and hence, by Theorem \ref{<->cs}, the operator $W_{\psi,\varphi,\max}$ is $\calC_{a,b,c}$-selfadjoint.

Thus, conclusion (1) follows from the following inclusions
$$
\calC_{a,b,c}W_{\psi,\varphi}\calC_{a,b,c}\preceq \calC_{a,b,c}W_{\psi,\varphi,\max}\calC_{a,b,c}=W_{\psi,\varphi,\max}^*\preceq W_{\psi,\varphi}^*=\calC_{a,b,c}W_{\psi,\varphi}\calC_{a,b,c}.
$$
\end{proof}

\section{Hermiticity}\label{s5}
Recall that a closed densely defined operator $T$ is said to be \emph{Hermitian} if $T=T^*$.  Cowen and Ko \cite{CK} found the exact structures when a weighted composition operator $W_{\psi,\varphi}$ is Hermitian on the Hardy space in the unit disk $\D$, under the additional assumption that  $\psi$ is bounded on $\D$. With the help of this assumption,  the operator $W_{\psi,\varphi}$ is certainly bounded on the Hardy space.

In this section, we investigate the Hermiticity of unbounded weighted composition operators acting on Fock space $\calF^2$. As in the previous section, we first consider maximal weighted composition operators and characterize these operators which are Hermitian. Then we use this characterization to show that the Hermiticity cannot be detached from the maximal domains.

A necessary condition for a maximal weighted composition operator to be Hermitian is provided by the following proposition.

\begin{prop}\label{hermitian-op-1}
Let $W_{\psi,\varphi,\max}$ be a maximal weighted composition operator induced by two entire functions $\psi$, $\varphi$. If the following identities
$$
W_{\psi,\varphi,\max}K_z=W_{\psi,\varphi,\max}^*K_z,\quad\forall z\in\C
$$
hold, then the symbols are of the following forms
\begin{equation}\label{s-psi-selfad}
\varphi(z)=Az+B,\quad\psi(z) = Ce^{\overline{B}z},\ \hbox{with $A\in\R,C\in\R\setminus\{0\}$, and $B\in\C$}.
\end{equation}
\end{prop}

\begin{proof}
For any $u,z\in\C$, we have
$$
(W_{\psi,\varphi,\max}K_z)(u)=(W_{\psi,\varphi,\max}^*K_z)(u),
$$
which means, by Lemma \ref{W*Kz-prop}(1), that
\begin{equation}\label{psi-phi}
\psi(u)e^{\varphi(u)\overline{z}}=\overline{\psi(z)}e^{u\overline{\varphi(z)}}.
\end{equation}
In particular, for $z=0$ we get $\psi(u)=\overline{\psi(0)}e^{u\overline{\varphi(0)}}$, and then with $u=0$, we obtain $C=\psi(0)\in\R\setminus\{0\}$. Then identity \eqref{psi-phi} becomes
$$
\overline{\psi(0)}e^{u\overline{\varphi(0)}}e^{\varphi(u)\overline{z}}=\psi(0)e^{\overline{z}\varphi(0)}e^{u\overline{\varphi(z)}},
$$
which gives
$$
\varphi(u)\overline{z}+u\overline{\varphi(0)}=u\overline{\varphi(z)}+\overline{z}\varphi(0),
$$
and so
$$
\frac{\varphi(u)-\varphi(0)}{u}=A\in\R.
$$
Thus, $\varphi(z)=Az+B$ with $\varphi(0)=B\in\C$.
\end{proof}

It turns out that condition \eqref{s-psi-selfad} is also the sufficient condition for a maximal weighted composition operator to be Hermitian.

\begin{thm}[Hermitian criterion]\label{hermitian-op}
Let $W_{\psi,\varphi,\max}$ be a maximal weighted composition operator induced by two entire functions $\psi$, $\varphi$ with $\psi\not\equiv 0$. Then the following assertions are equivalent.
\begin{enumerate}
\item The operator $W_{\psi,\varphi,\max}$ is Hermitian.
\item The operator $W_{\psi,\varphi,\max}$ is densely defined and it satisfies $W_{\psi,\varphi,\max}^*\preceq W_{\psi,\varphi,\max}$.
\item The symbols are of forms \eqref{s-psi-selfad}, that is
$$
\varphi(z)=Az+B,\quad\psi(z) = Ce^{\overline{B}z},\ \hbox{with $A\in\R,C\in\R\setminus\{0\}$, and $B\in\C$}.
$$
\end{enumerate}
\end{thm}

\begin{proof}
It is clear that $(1)\Longrightarrow(2)$, while implication $(2)\Longrightarrow(3)$ follows from Proposition \ref{hermitian-op-1}. It remains to prove that $(3)\Longrightarrow(1)$. 

Indeed, suppose that assertion (3) holds. Note that by Proposition \ref{W-closed}, the operator $W_{\psi,\varphi,\max}$ is always closed. A direct computation shows that kernel functions belong to the domain $\text{dom}(W_{\psi,\varphi,\max})$. Furthermore, by Theorem \ref{ad-2}, the operator $W_{\psi,\varphi,\max}$ is Hermitian.
\end{proof}
Like as the the complex symmetry, we also discover that there is no non-trivial domain for an unbounded weighted composition operator $W_{\psi,\varphi}$ on which $W_{\psi,\varphi}$ is Hermitian.

\begin{thm}\label{self-arbitrary-domain}
Let $W_{\psi,\varphi}$ be an unbounded weighted composition operator induced by the symbols $\psi$, $\varphi$ with $\psi\not\equiv 0$. Then it is Hermitian if and only if the following conditions hold.
\begin{enumerate}
\item $W_{\psi,\varphi}=W_{\psi,\varphi,\max}$.
\item The symbols are of forms \eqref{s-psi-selfad}, that is
\begin{equation*}
\varphi(z)=Az+B,\quad\psi(z) = Ce^{\overline{B}z},\ \hbox{with $A\in\R, C\in\R\setminus\{0\}$, and $B\in\C$}.
\end{equation*}
\end{enumerate}
\end{thm}
\begin{proof}
The sufficiency follows from Theorem \ref{hermitian-op}.

For the necessity, suppose that the operator $W_{\psi,\varphi}$ is Hermitian. First, we show that the operator $W_{\psi,\varphi,\max}$ is Hermitian.

Indeed, since $W_{\psi,\varphi}\preceq W_{\psi,\varphi,\max}$, by \cite[Proposition 1.6]{KS}, we have
$$
W_{\psi,\varphi,\max}^*\preceq W_{\psi,\varphi}^*=W_{\psi,\varphi}\preceq W_{\psi,\varphi,\max}.
$$
Lemma \ref{W*Kz-prop} shows that kernel functions always belong to the domain $\text{dom}(W_{\psi,\varphi,\max}^*)$, and so,
$$
W_{\psi,\varphi,\max}^*K_z(u)=W_{\psi,\varphi,\max}K_z(u),\quad\forall z,u\in\C.
$$
By Proposition \ref{hermitian-op-1}, the symbols are of forms \eqref{s-psi-selfad}, and hence, by Theorem \ref{hermitian-op}, the operator $W_{\psi,\varphi,\max}$ is Hermitian.

Thus, conclusion (1) follows from the following inclusions
$$
W_{\psi,\varphi}\preceq W_{\psi,\varphi,\max}=W_{\psi,\varphi,\max}^*\preceq W_{\psi,\varphi}^*=W_{\psi,\varphi}.
$$
\end{proof}

\section{Normality and cohyponormality}\label{s6}
Recall that a closed densely defined operator $T$ is called 
\begin{enumerate}
\item {\it normal} if $\text{dom}(T)=\text{dom}(T^*)$ and $\|Tx\|=\|T^*x\|,\ \forall x\in\text{dom}(T)$;
\item {\it cohyponormal} if $\text{dom}(T^*)\subseteq\text{dom}(T)$, $\|T^*x\|\geq\|Tx\|,\ \forall x\in\text{dom}(T^*)$.
\end{enumerate}
Note that a normal operator must be necessarily cohyponormal, but the inverse statement fails to holds. For a cohyponormal operator $T$, if $\lambda$ is an eigenvalue of the adjoint $T^*$, then $\overline{\lambda}$ is an eigenvalue of $T$.

The entire class of normal bounded weighted composition operators $W_{\psi,\varphi}$ on the Hardy space over $\D$ is still not well understood. Bourdon and Nayaran \cite{BN} characterized exactly the case when the symbol $\varphi$ has an interior fixed point. Later, Cowen, Jung and Ko \cite{CJK} discovered that when the symbol $\varphi$ has an interior fixed point, cohyponormality is equivalent to normality. These authors used the assumption that the symbol $\psi$ is bounded  on $\D$. The case when fixed points of $\varphi$ lie on the circle is difficult and remains unsolved completely. We refer the reader to the survey \cite{DT} for more details.

This situation on Fock space $\calF^2$ can be solved completely. Le \cite{TL} succeeded to characterize all \emph{bounded} normal weighted composition operators on $\calF^2$. It should be emphasized that his proof relies on the criteria (Proposition \ref{exaa}) for boundedness of weighted composition operators.  

In this section, we give complete descriptions of \emph{unbounded} weighted composition operators, which are cohyponormal as well as normal on $\calF^2$, with a different approach than that of Le.

The following technical lemma is needed in proving the sufficient conditions of the next two theorems.
\begin{lem}\label{dom-dom}
Let $\psi(z)=Ce^{Dz}$, $\varphi(z)=Az+B$, $\widehat{\psi}(z)=\overline{C}e^{\overline{B}z}$, and $\widehat{\varphi}(z)=\overline{A}z+\overline{D}$, where $A,B,C,D$ are complex constants, with $C\ne 0$. Then we have the following conclusions.

If $A\ne 0$ and $D=\frac{A\overline{B}-\overline{B}+D}{\overline{A}}$, then
\begin{equation*}
\int_\C|\widehat{\psi}(z)f(\widehat{\varphi}(z))|^2e^{-|z|^2}\;dV(z)=M\int_\C|\psi(z)f(\varphi(z))|^2e^{-|z|^2}\;dV(z),\quad\forall f\in\calF^2,
\end{equation*}
where
$$
M=e^{-\left|\frac{B-\overline{D}}{A}\right|^2+2\re\left(\frac{|B|^2-\overline{BD}}{\overline{A}}\right)}.
$$

In particular, 
\begin{enumerate}
\item if $A=1$, then $M=e^{|B|^2-|D|^2}$;
\item if $A\ne 1$ and $D=\overline{B}(1-A)(1-\overline{A})^{-1}$, then $M=1$.
\end{enumerate}
\end{lem}
\begin{proof}
Let $f\in\calF^2$. By the explicit forms of $\widehat{\psi}$ and $\widehat{\varphi}$, the left-hand-side integral is rewritten as
$$
\int_\C|\widehat{\psi}(z)f(\widehat{\varphi}(z))|^2e^{-|z|^2}\;dV(z)=|C|^2\int_\C|f(\overline{A}z+\overline{D})|^2e^{-|z|^2+2\re(\overline{B}z)}\;dV(z).
$$
Doing the change of variables $\overline{A}z+\overline{D}=Au+B$, with the note that
$$\dfrac{A\overline{B}-\overline{B}+D}{\overline{A}}=D,$$
and also
$$
-|z|^2+2\re(\overline{B}z)=-|u|^2+2\re(Du)-\left|\frac{B-\overline{D}}{A}\right|^2+2\re\left(\frac{|B|^2-\overline{BD}}{\overline{A}}\right),
$$
the left-hand-side integral above is rewritten as
\begin{eqnarray*}
\int_\C|\widehat{\psi}(z)f(\widehat{\varphi}(z))|^2e^{-|z|^2}\;dV(z)%
&=& M\cdot\int_\C|Ce^{Du}f(Au+B)|^2e^{-|u|^2}\;dV(u)\\
&=& M\cdot\int_\C|\psi(z)f(\varphi(z))|^2e^{-|z|^2}\;dV(z).
\end{eqnarray*}

Furthermore, notice that when $A=1$, we have
$$
|B-\overline{D}|^2=|B|^2+|D|^2-2\re\overline{DB}=|D|^2-|B|^2+2\re(|B|^2-\overline{DB}),
$$
while for $A\ne 1$, since $\overline{D}=\frac{B(1-\overline{A})}{1-A}$, we have
$$
\frac{|B|^2-\overline{BD}}{\overline{A}}=\left|\frac{B-\overline{D}}{A}\right|^2\frac{A(1-\overline{A})}{A-\overline{A}},
$$
which implies, as $\re\left(\frac{A(1-\overline{A})}{A-\overline{A}}\right)=\frac{1}{2}$, that $M=1$.
\end{proof}

As in the previous sections, our first task is to characterize all maximal weighted composition operators, which are cohyponormal (Theorem \ref{cohypo-nor-equiv}) and normal (Theorem \ref{cohypo-nor-equiv-}).
\begin{thm}\label{cohypo-nor-equiv}
Let $W_{\psi,\varphi,\max}$ be a maximal weighted composition operator, induced by entire functions $\varphi, \psi$ with $\psi\not\equiv 0$. The following assertions are equivalent.
\begin{enumerate}
\item The operator $W_{\psi,\varphi,\max}$ is cohyponormal.
\item One of the following cases occurs:
\begin{enumerate}
\item $\varphi(z)=Az+B$, with $A\ne 1$, and $\psi(z)=Ce^{Dz}$, where
$$D=\overline{B}(1-A)(1-\overline{A})^{-1}.$$
\item $\varphi(z)=z+B$, and $\psi(z)=Ce^{Dz}$, where $|B|\geq |D|$.
\end{enumerate}
\end{enumerate}
\end{thm}

\begin{proof}
$\bullet$ We prove implication $(1)\Longrightarrow(2)$. Indeed, suppose that the operator $W_{\psi,\varphi,\max}$ is cohyponormal, and hence, by Theorem \ref{varphi-cohypo}(2) (for $S=I$, the identity operator), $\varphi(z)=Az+B$. There are two possibilities for $A$.

If $A=1$, then again by Theorem \ref{varphi-cohypo}(3), we immediately obtain (2b).

If $A\ne 1$, then for $d=B/(1-A)$, a fixed point of $\varphi$, Lemma \ref{W*Kz-prop}(1) gives $W_{\psi,\varphi,\max}^*K_d=\overline{\psi(d)}K_{\varphi(d)}=\overline{\psi(d)}K_d$. Thus, $K_d$ is an eigenvector of $W_{\psi,\varphi,\max}^*$ corresponding to the eigenvalue $\overline{\psi(d)}$.

This notice allows us to get $W_{\psi,\varphi,\max}K_d=\psi(d)K_d$. Consequently, taking into account the structure of the operator $W_{\psi,\varphi,\max}$, we get $\psi(z)K_d(\varphi(z))=\psi(d)K_d(z)$, which gives
$$
\psi(z)=\psi(d)e^{(1-A)\overline{d}z-B\overline{d}},\ \forall z\in\C.
$$
In particular, for $z=0$, we obtain $\psi(0)=\psi(d)e^{-B\overline{d}}$ and hence (2a) follows.

$\bullet$ We prove the inverse implication $(2)\Longrightarrow(1)$. Suppose that the functions $\varphi,\psi$ are of forms (2a-2b). We apply Theorem \ref{ad-2} to get $W_{\psi,\varphi,\max}^*=W_{\widehat{\psi},\widehat{\varphi},\max}$, where $\widehat{\varphi}(z)=\overline{A}z+\overline{D}$, $\widehat{\psi}(z)=\overline{C}e^{z\overline{B}}$. So, to get assertion (1), we have to show that
$$
\text{dom}(W_{\widehat{\psi},\widehat{\varphi},\max})\subseteq\text{dom}(W_{\psi,\varphi,\max}),\quad \|W_{\widehat{\psi},\widehat{\varphi},\max}f\|\geq\|W_{\psi,\varphi,\max}f\|,\,\forall f\in\text{dom}(W_{\widehat{\psi},\widehat{\varphi},\max}).
$$

If $A=0$, then by Proposition \ref{exaa}, the operator $W_{\psi,\varphi,\max}$ is bounded on $\calF^2$, and hence, by \cite[Theorem 3.3]{TL}, it must be normal.

Now consider the case $A\ne 0$. Let $g\in \text{dom}(W_{\widehat{\psi},\widehat{\varphi},\max})$, that is $\|W_{\widehat{\psi},\widehat{\varphi},\max}g\|<\infty$. By Lemma \ref{dom-dom}, we have $\|W_{\widehat{\psi},\widehat{\varphi},\max}g\|\geq\|W_{\psi,\varphi,\max}g\|$, which gives $g\in\text{dom}(W_{\psi,\varphi,\max})$. The proof of the theorem is complete.
\end{proof}

\begin{thm}\label{cohypo-nor-equiv-}
Let $W_{\psi,\varphi,\max}$ be a maximal weighted composition operator induced by entire functions $\varphi, \psi$ with $\psi\not\equiv 0$. The following assertions are equivalent.
\begin{enumerate}
\item The operator $W_{\psi,\varphi,\max}$ is normal.
\item One of the following cases occurs:
\begin{enumerate}
\item $\varphi(z)=Az+B$, with $A\ne 1$, and $\psi(z)=Ce^{Dz}$, where
$$D=\overline{B}(1-A)(1-\overline{A})^{-1}.$$
\item $\varphi(z)=z+B$, and $\psi(z)=Ce^{Dz}$, where $|B|=|D|$.
\end{enumerate}
\end{enumerate}
\end{thm}
\begin{proof}
$\bullet$ We prove implication $(1)\Longrightarrow(2)$. Suppose that the operator $W_{\psi,\varphi,\max}$ is normal, and hence, by Theorem \ref{cohypo-nor-equiv}, $\varphi(z)=Az+B$, $\psi(z)=Ce^{Dz}$, where
\begin{equation*}
D=
\begin{cases}
\overline{B}(1-A)(1-\overline{A})^{-1},\quad\text{if $A\ne 1$},\\
\text{a constant with $|D|\leq |B|$,\quad if $A=1$}.
\end{cases}
\end{equation*}
It remains to show that $|D|=|B|$ in the case when $A=1$. 

Indeed, in this case, a direct computation shows that $W_{\psi,\varphi,\max}K_z=Ce^{B\overline{z}}K_{z+\overline{D}}$, and hence,
$$
\|W_{\psi,\varphi,\max}K_z\|^2=|C|^2e^{2\re(z\overline{B})+|z+\overline{D}|^2}=|C|^2e^{|z|^2+2\re[z(\overline{B}+D)]+|D|^2}.
$$
On the other hand, by Lemma \ref{W*Kz-prop}, we have $W_{\psi,\varphi,\max}^*K_z=\overline{Ce^{Dz}}K_{z+B}$, which gives
$$
\|W_{\psi,\varphi,\max}^*K_z\|^2=|C|^2e^{2\re(Dz)+|z+B|^2}=|C|^2e^{|z|^2+2\re[z(\overline{B}+D)]+|B|^2}.
$$
Since the operator $W_{\psi,\varphi,\max}$ is normal, we must have $\|W_{\psi,\varphi,\max}K_z\|=\|W_{\psi,\varphi,\max}^*K_z\|$, or equivalently
$$
|C|^2e^{|z|^2+2\re[z(\overline{B}+D)]+|D|^2}=|C|^2e^{|z|^2+2\re[z(\overline{B}+D)]+|B|^2},\quad\forall z\in\C.
$$
In particular with $z=0$, we get $|D|=|B|$.

$\bullet$ We prove the inverse implication $(2)\Longrightarrow(1)$. Suppose that the functions $\varphi,\psi$ are of forms (2a-2b). Note that $W_{\psi,\varphi,\max}^*=W_{\widehat{\psi},\widehat{\varphi},\max}$, where $\widehat{\varphi}(z)=\overline{A}z+\overline{D}$, $\widehat{\psi}(z)=\overline{C}e^{z\overline{B}}$. By arguments similar to those used in the proof of Theorem \ref{cohypo-nor-equiv}, we can show that
$$
\text{dom}(W_{\widehat{\psi},\widehat{\varphi},\max})\subseteq\text{dom}(W_{\psi,\varphi,\max}),\quad \|W_{\widehat{\psi},\widehat{\varphi},\max}f\|\geq\|W_{\psi,\varphi,\max}f\|,\,\forall f\in\text{dom}(W_{\widehat{\psi},\widehat{\varphi},\max}).
$$

Let $h\in\text{dom}(W_{\psi,\varphi,\max})$. Then by Lemma \ref{dom-dom}, $h\in\text{dom}(W_{\widehat{\psi},\widehat{\varphi},\max})$. Thus $\text{dom}(W_{\psi,\varphi,\max})=\text{dom}(W_{\widehat{\psi},\widehat{\varphi},\max})$, and the proof of the theorem is complete.
\end{proof}

\begin{rem}
Comparing \cite[Theorem 3.3]{TL} and our Theorem \ref{cohypo-nor-equiv-}, it can be seen an essential difference between bounded and unbounded cases; namely, if $A=1$, then for the first case the constant $D$ is exactly $-\overline{B}$, while the second case only requires $|D|=|B|$. This difference is due to the use of Proposition \ref{exaa}. Thus, Theorem \ref{cohypo-nor-equiv-} contains the corresponding result in \cite{TL} as a particular case.
\end{rem}

It turns out that unbounded normal weighted composition operators on Fock space $\calF^2$ are contained properly in the $\calC_{a,b,c}$-selfadjoint class. This remark is demonstrated in the below corollary.

\begin{cor}\label{cohypo-nor-equiv-cor}
Let $\psi$ and $\varphi$ be two entire functions such that $W_{\psi,\varphi,\max}$ is normal on $\calF^2$ (that is, the symbols satisfy Theorem \ref{cohypo-nor-equiv-}(2)). Then the operator $W_{\psi,\varphi,\max}$ is $\calC$-selfadjoint with respect to the weighted composition conjugation $\calC_{a,b,c}$ given by
\begin{equation}\label{choice-normal}
b=0,\quad c=1,\quad a=
\begin{cases}
\overline{B}(1-A)B^{-1}(1-\overline{A})^{-1},\quad\text{if $B\ne 0, A\ne 1$},\\
DB^{-1},\quad\text{if $B\ne 0,A=1$},\\
1,\quad\text{if $B=0$}.
\end{cases}
\end{equation}
\end{cor}

\begin{proof}
The complex numbers $a,b,c$ given by \eqref{choice-normal} satisfy condition \eqref{abc-cond}, and hence, the operator $\calC_{a,b,c}$ is a conjugation. Moreover, condition \eqref{s-psi-2} is satisfied, and so by Theorem \ref{<->cs}, the operator $W_{\psi,\varphi,\max}$ is $\calC_{a,b,c}$-symmetric.
\end{proof}

We are ready to discover that that normality and cohyponormality cannot be separated from maximal domains. In other words, a normal (or cohyponormal) weighted composition operator must necessarily be maximal.
\begin{thm}\label{cohypo-nor-equiv--}
Let $W_{\psi, \varphi}$ be an unbounded weighted composition operator on $\calF^2$, induced by entire functions $\varphi, \psi$ with $\psi\not\equiv 0$. The following assertions are equivalent.
\begin{enumerate}
\item The operator $W_{\psi, \varphi}$ is cohyponormal.
\item The operator $W_{\psi, \varphi}$ satisfies
\begin{enumerate}
\item $W_{\psi, \varphi}=W_{\psi, \varphi,\max}$.
\item One of the following cases occurs:
\begin{enumerate}
\item $\varphi(z)=Az+B$, with $A\ne 1$, and $\psi(z)=Ce^{Dz}$, where
$$D=\overline{B}(1-A)(1-\overline{A})^{-1}.$$
\item $\varphi(z)=z+B$, and $\psi(z)=Ce^{Dz}$, where $|B|\geq |D|$.
\end{enumerate}
\end{enumerate}
\end{enumerate}
\end{thm}
\begin{proof}
$\bullet$ We prove implication $(1)\Longrightarrow(2)$. Suppose that the operator $W_{\psi, \varphi}$ is cohyponormal, which means
$$
\text{dom}(W_{\psi,\varphi}^*)\subseteq\text{dom}(W_{\psi,\varphi}),\quad\|W_{\psi,\varphi}^*f\|\geq\|W_{\psi,\varphi}f\|,\,\forall f\in\text{dom}(W_{\psi,\varphi}^*).
$$
Since $W_{\psi,\varphi}\preceq W_{\psi,\varphi,\max}$, we have
\begin{equation}\label{point-add}
\text{dom}(W_{\psi,\varphi,\max}^*)\subseteq\text{dom}(W_{\psi,\varphi}^*)\subseteq\text{dom}(W_{\psi,\varphi})\subseteq\text{dom}(W_{\psi,\varphi,\max}).
\end{equation}
Let $f\in\text{dom}(W_{\psi,\varphi,\max}^*)$. Also since $W_{\psi,\varphi}\preceq W_{\psi,\varphi,\max}$, we have $W_{\psi,\varphi,\max}^*\preceq W_{\psi,\varphi}^*$, which gives
$$
\|W_{\psi,\varphi,\max}^*f\|=\|W_{\psi,\varphi}^*f\|\geq \|W_{\psi,\varphi}f\|=\|W_{\psi,\varphi,\max}f\|.
$$
Thus, the operator $W_{\psi,\varphi,\max}$ is cohyponormal, and hence, by Theorem \ref{cohypo-nor-equiv}, we get assertion (2b).

To prove (2a), we note that $W_{\psi,\varphi,\max}^*=W_{\widehat{\psi},\widehat{\varphi},\max}$, where $\widehat{\varphi}(z)=\overline{A}z+\overline{D}$, $\widehat{\psi}(z)=\overline{C}e^{z\overline{B}}$. By Lemma \ref{dom-dom} and Remark \ref{f-in-dom}, we see that
$$\text{dom}(W_{\psi,\varphi,\max})=\text{dom}(W_{\widehat{\psi},\widehat{\varphi},\max})=\text{dom}(W_{\psi,\varphi,\max}^*),$$
and so, by \eqref{point-add}, we conclude that
$$
\text{dom}(W_{\psi,\varphi})=\text{dom}(W_{\psi,\varphi,\max}).
$$
Thus, we must have $W_{\psi,\varphi,\max}=W_{\psi,\varphi}$, which gives conclusion (2a).

$\bullet$ The implication $(2)\Longrightarrow(1)$ follows from Theorem \ref{cohypo-nor-equiv}.
\end{proof}

\begin{thm}\label{normal-2-arbi}
Let $W_{\psi,\varphi}$ be an unbounded weighted composition operator, induced by two entire functions $\varphi, \psi$ with $\psi\not\equiv 0$. Then the followings are equivalent.
\begin{enumerate}
\item The operator $W_{\psi,\varphi}$ is normal
\item The operator $W_{\psi,\varphi}$ satisfies
\begin{enumerate}
\item $W_{\psi,\varphi}=W_{\psi,\varphi,\max}$.
\item One of the following cases occurs:
\begin{enumerate}
\item $\varphi(z)=Az+B$, with $A\ne 1$, and $\psi(z)=Ce^{Dz}$, where $$D=\overline{B}(1-A)(1-\overline{A})^{-1}.$$
\item $\varphi(z)=z+B$, and $\psi(z)=Ce^{Dz}$, where $|D|=|B|$.
\end{enumerate}
\end{enumerate}
\end{enumerate}
\end{thm}
\begin{proof}
By Theorem \ref{cohypo-nor-equiv-}, we have $(2)\Longrightarrow (1)$, while implication $(1)\Longrightarrow (2)$ follows from Theorems \ref{cohypo-nor-equiv-} and \ref{cohypo-nor-equiv--} (since a normal operator is always coyhyponormal).
\end{proof}


\bibliographystyle{plain}
\bibliography{refs}
\end{document}